\documentclass[pdflatex,sn-vancouver-num]{sn-jnl}% Math and Physical Sciences Numbered Reference Style
\usepackage{amsmath,amssymb,amsfonts}%
\usepackage{amsthm}%
\usepackage[title]{appendix}%
\newtheorem{theorem}{Theorem}%  meant for continuous numbers
\theoremstyle{thmstyletwo}%
\newtheorem{example}{Example}%
\newtheorem{question}{Question}

\theoremstyle{thmstylethree}%
\newtheorem{lemma}{Lemma}
\DeclareMathOperator*{\const}{const}
\DeclareMathOperator*{\coef}{coef}

\newcommand{\binomsq}[2]{\genfrac{[}{]}{0pt}{}{#1}{#2}}
\newcommand{\Z}{\mathbb{Z}}

\begin{document}

\title[Magic positivity]{On the magic positivity of the Ehrhart polynomials of dilated base polytope of minimal matroid}


\author[1]{\fnm{Jiaxuan} \sur{Chen}}\email{jiaxchen@ohs.stanford.edu}

\author[2]{\fnm{Chance} \sur{Crigler}}\email{ccrigler@calpoly.edu}

\author[1]{\fnm{Shengtao} \sur{Ding}}\email{alexding@ohs.stanford.edu}

\author*[1,2]{\fnm{Dana} \sur{Paquin}}\email{dpaquin@stanford.edu}

\author[1]{\fnm{Julian} \sur{Schennach}}\email{julian.vs.usa@gmail.com}

\affil[1]{\orgname{Stanford Online High School, Stanford University}, \orgaddress{\street{415 Broadway, 2nd floor}, \city{Redwood City}, \postcode{94063}, \state{CA}, \country{USA}}}

\affil[2]{\orgdiv{Department of Mathematics}, \orgname{California Polytechnic State University}, \orgaddress{\street{1 Grand Avenue}, \city{San Luis Obispo}, \postcode{93407}, \state{CA}, \country{USA}}}


\abstract{In this paper, we consider the magic positivity of the Ehrhart polynomial of dilated base polytope of minimal matroid of rank $k$ dimension $n$. We prove that for $\displaystyle 1\leq k<\lfloor\frac{n}{2}\rfloor$, the $n-k-1$ dilate of the Ehrhart polynomial is not magic positive. We also show that for sufficiently large $n$, the $n-k$ dilate of the Ehrhart polynomial is, in general, magic positive. This gives rise to a lower bound of the $m-$index of the Ehrhart polynomial.}




\maketitle

\section{Introduction}

A \emph{lattice point} is an element of $\Z^d$, and a \emph{lattice
polytope} is the convex hull of finitely many lattice points.  Given a lattice
polytope $P$, let
\[
  L_P(t) := \bigl|tP \cap \Z^d\bigr| \qquad (t \in \Z_{\ge 0})
\]
denote the number of lattice points in the $t$th dilate of $P$.  A theorem of
Ehrhart states that $L_P(t)$ agrees with a polynomial in $t$, called the
\emph{Ehrhart polynomial} of $P$ \cite{ehrhart1962polyedres, beck:disc}.

It is often useful to encode these values in the generating function, called the
\emph{Ehrhart series},
\[
  E(x) := \sum_{t \ge 0} L_P(t)\, x^t.
\]
A fundamental result in Ehrhart theory states that this series is a rational
function whose denominator depends only on $\dim P$; equivalently,
\[
  \sum_{t \ge 0} L_P(t)\, x^t \;=\; \frac{h^*(x)}{(1-x)^{\dim P+1}}
\]
for a polynomial $h^*(x)$, called the \emph{$h^*$-polynomial} (or the numerator
of the Ehrhart series).

A central question in Ehrhart theory is to understand the structure of
$h^*(x)$.  One approach is to analyze the behavior of its coefficients,
particularly with respect to three closely related properties.  If
$f(x) = \sum_{i=0}^d a_i x^i$ is a polynomial of degree $d$ with $a_i \in \mathbb{R}$,
then:
\begin{itemize}
  \item $f(x)$ is \emph{unimodal} if there exists $k$ such that
    $a_0 \le a_1 \le \cdots \le a_k \ge \cdots \ge a_d$;
  \item $f(x)$ is \emph{log-concave} if $a_i^2 \ge a_{i-1}a_{i+1}$ for all $i$;
  \item $f(x)$ is \emph{real-rooted} if all its roots are real.
\end{itemize}
It is well known that real-rootedness implies log-concavity, which in turn
implies unimodality.  These properties are of significant interest for
$h^*$-polynomials of lattice polytopes, though they can be difficult to verify
directly.

One novel approach to establishing such properties is via \emph{magic
positivity}.  Any polynomial $f(x)$ of degree $d$ admits a unique expansion in
the basis $\{x^i(x+1)^{d-i}\}_{i=0}^d$:
\[
  f(x) = \sum_{i=0}^d c_i\, x^i(x+1)^{d-i}.
\]
We say that $f$ is \emph{magic positive} if $c_i \ge 0$ for all $0 \le i \le d$.

\cite{konoike2025magicpositivityehrhartpolynomials}

It was shown in \cite{konoike2025magicpositivityehrhartpolynomials} that if $f = L_P$ is the Ehrhart polynomial of a
lattice polytope $P$, then magic positivity of $f$ implies real-rootedness of
the $h^*$-polynomial of $P$, and hence also log-concavity and unimodality of
$h^*(x)$.  Magic positivity of the Ehrhart polynomial is therefore a powerful
sufficient condition for the favorable coefficient behavior of the $h^*$-polynomial.

A key result of Konoike \cite{konoike2025magicpositivityehrhartpolynomials} shows that if $P$ is an Ehrhart-positive
polytope, then the Ehrhart polynomial of the $m$-th dilate $mP$ is magic
positive for all sufficiently large $m$.  This guarantees that the following
invariant is well-defined:
\[
  m\text{-index}(P) := \min\bigl\{\, k \in \Z_{>0} :
  L_{kP}(t) \text{ is magic positive}\,\bigr\}.
\]
Studying the $m$-index provides a natural way to quantify how much dilation is
required before magic positivity is attained.  In this paper, we focus on a
particular family of polytopes.

The base polytope of the minimal matroid of rank $k$ and dimension $n$,
denoted $B(T_{k,n})$, has lattice point
enumerator \cite{konoike2025magicpositivityehrhartpolynomials}
\[
  D_{k,n}(x) := \frac{1}{\binom{n-1}{k-1}}
  \binom{x+n-k}{n-k}
  \sum_{j=0}^{k-1} \binom{n-k-1+j}{j} \binom{x+j}{j}.
\]
Konoike \cite{konoike2025magicpositivityehrhartpolynomials} computed the
$m$-index of $B(T_{k,n})$ in the cases $k=1$ and $k=2$, obtaining
\[
  m\text{-index}(B(T_{1,n})) = n-1 \qquad\text{and}\qquad
  m\text{-index}(B(T_{2,n})) = n-2,
\]
and on the basis of these examples posed the following question.

\begin{question}[\cite{konoike2025magicpositivityehrhartpolynomials},
  Question~5.5]
  For $1 \le k \le \lfloor n/2 \rfloor$, do we have
  \[
    m\text{-index}(B(T_{k,n})) = n-k\,?
  \]
\end{question}

Although we do not resolve this question in full generality, we obtain
several results that provide significant evidence for an affirmative answer and
establish the conjecture in important cases.  Our main results are as follows.

\begin{theorem}\label{thm:n-k-1}
  For $1 \le k < \lfloor n/2 \rfloor$, the $m$-index of $B(T_{k,n})$ is at
  least $n-k$.  Equivalently, the $(n-k-1)$-th dilate of $D_{k,n}$ is not
  magic positive.
\end{theorem}

Theorem~\ref{thm:n-k-1} shows that the conjectured value $n-k$ cannot be
lowered by one in the range $1 \le k < \lfloor n/2 \rfloor$: the
$(n-k-1)$-th dilate always fails magic positivity.

\begin{theorem}\label{thm: asymptotic}
  For fixed $k$, $D_{k,n}((n-k)x)$ is magic positive for all sufficiently
  large $n$.
\end{theorem}

Theorem~\ref{thm: asymptotic} establishes the conjecture asymptotically: for
any fixed rank $k$, the $m$-index equals $n-k$ for all large enough $n$.

\begin{theorem}\label{thm: sharp}
  For $n = 3, 4, \ldots, 20$, the polynomial $D_{k,n}((n-k)x)$ is magic
  positive when $k = \lfloor n/2 \rfloor$, but fails to be magic positive when
  $k = \lfloor n/2 \rfloor + 1$.
\end{theorem}

Theorem~\ref{thm: sharp} shows that the range $1 \le k \le \lfloor n/2 \rfloor$
in the conjecture is sharp: magic positivity at the $(n-k)$-th dilate does not
hold in general for $k > \lfloor n/2 \rfloor$.

The paper is organized as follows.  In Section~\ref{secnk1} we prove
Theorem~\ref{thm:n-k-1} by showing that the coefficient of $x^{n-2}(x+1)$ in
the magic-positive-basis expansion of $D_{k,n}((n-k-1)x)$ is strictly negative.
The key tools are a factorization of $D_{k,n}$ into components $Q_{k,n}$ and
$S_{k,n}$, a combinatorial analysis of $Q_{k,n}$ via its product structure, and
a careful study of the leading coefficient of $S_{k,n}$ using Vandermonde's
identity and a symmetry argument involving the Stirling numbers of the first
kind.  In Section~3 we prove Theorem~\ref{thm: asymptotic} by establishing that
all magic-positive-basis coefficients of both $Q_{k,n}((n-k)x)$ and
$S_{k,n}((n-k)x)$ are eventually positive as $n \to \infty$, using asymptotic
estimates on elementary symmetric functions and unsigned Stirling numbers.
Section~4 presents the numerical verification underlying Theorem~\ref{thm: sharp},
carried out using Maple (version~2024).

\section{The $n-k-1$ dilate of $D_{k,n}$}
\label{secnk1}
	In this section, we prove the non-magic-positivity of the Ehrhart polynomial of the $n-k-1$ dilate of the base polytope of minimal matroid of dimension $n$, rank $k$, proving Theorem \ref{thm:n-k-1} by deriving a lower bound on the $m-$index. We proceed by showing that, for $1 \leq k < \left\lfloor \tfrac{n}{2} \right\rfloor$, $D_{k,n}(x)$ does not exhibit magic positivity at the $(n-k-1)$ dilate because the coefficient of $x^{n-2}(x+1)$ is negative.

%JS this display seems redundant with Theorem 1. At the end of the section, we say that we prove Theorem 1, not proposition 4. So I moved the displayed proposition as part of the running text above.
%	\begin{proposition}
%	For
%	$1 \leq k < \left\lfloor \tfrac{n}{2} \right\rfloor$, $D_{k,n}(x)$ does not exhibit magic positivity at the $(n-k-1)$ dilate because the coefficient of $x^{n-2}(x+1)$ is negative.
%	\end{proposition}

    We wish to study the signs of the coefficients of $D_{k,n}((n-k-1)x)$ in the basis $\{ x^{i}(x+1)^{d-i} \}_{i=0}^{d}$. 
	To aid in this, we rewrite $D_{k,n}$ in the following form:
    %JS: rewrote this in a shorter more "paper like" form
    %also renamed P_{k,n} as Q_{k,n} to use the same convention as later
    \[
    D_{k,n}\!\big((n-k-1)x\big) 
		\propto Q_{k,n}((n-k-1)x)\ S_{k,n}((n-k-1)x)
    \]
	%where
%	\begin{align*}
%		D_{k,n}\!\big((n-k-1)x\big) 
%		&\propto 
%		\binom{(n-k-1)x+n-k}{n-k} 
%		\sum_{j=0}^{k-1} 
%		\binom{n-k-1+j}{j} \binom{(n-k-1)x+j}{j} \\[1ex]
%		&\propto
%		\left(\prod_{i=0}^{n-k-1} \big((n-k-1)x + n-k - i\big)\right)
%		\sum_{j=0}^{k-1} 
%		\binom{n-k-1+j}{j} \binom{(n-k-1)x+j}{j} \\[1ex]
%		&= 
%		\left(\prod_{i=0}^{n-k-1} \big((n-k-i)(x+1) + (i-1)x\big)\right)
%		\sum_{j=0}^{k-1} 
%		\binom{n-k-1+j}{j} \binom{(n-k-1)x+j}{j}
%	\end{align*}
	%\begin{eqnarray*}
	%Q_{k,n}((n-k-1)x)&:=&\binom{(n-k-1)x+n-k}{n-k}\\
    %&=&\prod_{i=0}^{n-k-1} \big((n-k-1)x + n-k - i\big)\\
    %&=&\prod_{i=0}^{n-k-1} \big((n-k-i)(x+1) + (i-1)x\big) \text{ and}\\
	%S_{k,n}((n-k-1)x)&:=&\sum_{j=0}^{k-1}\binom{n-k-1+j}{j} \binom{(n-k-1)x+j}{j}.
	%\end{eqnarray*}
    where the proportionality constant is always positive and where
	\begin{eqnarray}
	Q_{k,n}(x)&:=&\binom{x+n-k}{n-k}  \label{defQkn} \text{ and}\\
	S_{k,n}(x)&:=&\sum_{j=0}^{k-1}\binom{n-k-1+j}{j} \binom{x+j}{j}. \label{defSkn}
	\end{eqnarray}
	We are then interested in studying the coefficient of $x^{n-2}(x+1)$ of $Q_{k,n}((n-k-1)x) \cdot S_{k,n}((n-k-1)x)$ after expansion in the magic basis.
	
	\medskip
	To do this, we first note that
	\begin{eqnarray*}    
	Q_{k,n}((n-k-1)x)&=&\prod_{i=0}^{n-k-1} \big((n-k-1)x + n-k - i\big) \\
    &=&\prod_{i=0}^{n-k-1} \big((n-k-i)(x+1) + (i-1)x\big)
    \end{eqnarray*}
	and expand $Q_{k,n}((n-k-1)x)$ in the basis $\{x^{a}(x+1)^{\,n-k-a}\}_{a=0}^{n-k}$ and $S_{k,n}((n-k-1)x)$ in the basis
	$\{x^{b}(x+1)^{\,k-1-b}\}_{b=0}^{k-1}$. Since
	\[
	x^{a}(x+1)^{\,n-k-a}\cdot x^{b}(x+1)^{\,k-1-b}
	= x^{a+b}(x+1)^{\,n-1-(a+b)},
	\]
	the coefficient of $x^{n-2}(x+1)$ in $Q_{k,n}((n-k-1)x)S_{k,n}((n-k-1)x)$ can only come from pairs with $a+b=n-2$.
	So, it depends only on the
	$x^{n-k}$ and $x^{n-k-1}(x+1)$ terms of $Q_{k,n}(n-k-1(x))$ and the $x^{k-1}$ and $x^{k-2}(x+1)$ terms of $S_{k,n}((n-k-1)x)$.
	
    \begin{lemma}
	When $Q_{k,n}((n-k-1)x)$ is expanded in the magic positive basis, the coefficient of $x^{(n-k)}$ is 0 while the coefficient of $x^{n-k-1}(x+1)$ is always negative.
	\end{lemma}
\begin{proof}

	To obtain a pure $x^{n-k}$ term from
	 \[
	 Q_{k,n}((n-k-1)x)=\prod_{i=0}^{n-k-1}\big((n-k-i)(x+1)+(i-1)x\big),
	 \]
	 we would have to choose the $x$-summand from every factor. However, the $i=1$ factor is
	 \[
	 (n-k-1)(x+1)+0\cdot x,
	 \]
	 so it contributes no $x$-term. So $x^{n-k}$ cannot occur, and its coefficient is $0$.
	 
  To obtain
	 $x^{n-k-1}(x+1)$ from
	 \[
	 Q_{k,n}((n-k-1)x)=\prod_{i=0}^{n-k-1}\big((n-k-i)(x+1)+(i-1)x\big),
	 \]
	 we must choose an $(x+1)$-summand from exactly one factor and an $x$-summand from the remaining $n-k-1$ factors.
	 Since the $i=1$ factor equals $(n-k-1)(x+1)+0\cdot x$, it has no $x$-summand. So the $(x+1)$ must come from this $i=1$ factor, and we must choose $x$ from every $i\neq 1$ factor.
	 This contribution's sign must be negative since the $i=0$ factor contributes $-x$ while all other chosen $x$-summands (with $i\ge 2$) have positive coefficients.
	\end{proof}
	 We now move our attention to $S_{k,n}((n-k-1)x)$, which we calculate via the following slightly more general lemma.

%JS: one part of the paper used \alpha and another used \beta. I converted all to \alpha.

%JS: This lemma combines what was Lemma 2 and 6 into a single lemma that is referenced in both sections. (Otherwise it was repetitive.)
\begin{lemma}
\label{lemdilS}
For a given dilation $\sigma \in\mathbb{N}$, the expansion of $S_{k,n}(\sigma x)$ in the magic positive basis is
\begin{equation}
\label{eqdilS}
S_{k,n}(\sigma x)=\sum_{r=0}^{k-1}\alpha_r x^r (x+1)^{k-1-r}
\end{equation}
where \[\alpha_r=\sum_{j=0}^{k-1}\binom{n-k-1+j}{j}\frac{1}{j!}\sum_{i=0}^{j}s(j,i)\sum_{t=0}^{\min(i,r)}\binom{i}{t}\sigma^t j^{i-t} (-1)^{r-t}\binom{k-1-t}{r-t},\]
in which the $\min(i,r)$ can be replaced by $i$ under the convention that $\binom{a}{b}=0$ for $b<0$.
\end{lemma}
\begin{proof}
We use the identities
	 \[
	 \sum_{i=t}^{d} (-1)^{\,i-t} \binom{d-t}{\,i-t\,} x^{i}(x+1)^{d-i} \;=\; x^{t}  \text{ and }
	 \binom{x}{j} \;=\; \frac{1}{j!}\sum_{i=0}^{j} s(j,i)\,x^{i},
	 \]
	 where the former is readily obtained by specializing Konoike's lemma 2.1 and the latter, from the properties of 
     signed Stirling numbers of the first kind $s(j,i)$ \cite[Chapter 6]{knuth:concrete}.
     Applying these identities and the Binomial Theorem to our sum yields:
\begin{align*}
\lefteqn{{S_{k,n}(\sigma x)}=\sum_{j=0}^{k-1}\binom{n-k-1+j}{j}\binom{\sigma x+j}{j}=\sum_{j=0}^{k-1}\binom{n-k-1+j}{j}\frac{1}{j!}\sum_{i=0}^j s(j,i)(\sigma x+j)^i}\\
&=\sum_{j=0}^{k-1}\binom{n-k-1+j}{j}\frac{1}{j!}\sum_{i=0}^j s(j,i)\sum_{t=0}^{i}\binom{i}{t}\sigma^t x^t j^{i-t}\\
&=\sum_{j=0}^{k-1}\binom{n-k-1+j}{j}\frac{1}{j!}\sum_{i=0}^j s(j,i)\sum_{t=0}^{i}\binom{i}{t}\sigma^t j^{i-t}\sum_{r=t}^{k-1} (-1)^{r-t}\binom{k-1-t}{r-t}x^r (x+1)^{k-1-r}\\
&=\sum_{r=0}^{k-1}\left(\sum_{j=0}^{k-1}\binom{n-k-1+j}{j}\frac{1}{j!}\sum_{i=0}^{j}s(j,i)\sum_{t=0}^{\min(i,r)}\binom{i}{t}\sigma^t j^{i-t} (-1)^{r-t}\binom{k-1-t}{r-t}\right)\times \\
&\quad x^r (x+1)^{k-1-r}\\
&=\sum_{r=0}^{k-1}\alpha_r x^r (x+1)^{k-1-r}.
\end{align*}
\end{proof}

%JS this result is not directly used later and is redundant with the text
%\begin{lemma}
%\label{lemmaSnk1}
%	 The coefficient of $x^{k-1}$ of $S_{k,n}((n-k-1)x)$ in the magic positive basis is always positive.
%	\end{lemma}

%JS here I tried to re-arrange to avoid a lemma within a lemma...
	 To study the sign of $\alpha_{k-1}$ we establish the following fact:
	\begin{lemma}\label{lemalphakm1} When $n$ is an integer and $1 \leq n \leq k$ we have $\alpha_{k-1}=(-1)^{n+1}(k-n-1).$
	\end{lemma}
\begin{proof}
By Lemma \ref{lemdilS} with $\sigma =n-k-1$, the $\alpha_{k-1}$ coefficient in Equation (\ref{eqdilS}) is given by
	 \[
	 \alpha_{k-1}
	 = \sum_{j=0}^{k-1} \binom{n-k-1+j}{j}\,\frac{1}{j!}
	 \sum_{i=0}^{j} s(j,i)
	 \sum_{t=0}^{\min(i,k-1)}
	 \binom{i}{t}\,(n-k-1)^{t}\,j^{\,i-t}\,
	 (-1)^{\,k-1-t}\binom{k-1-t}{\,k-1-t\,}.
	 \]
	 Since $i\leq k-1$, we obtain 
	 \[
	 \alpha_{k-1} = \sum_{j=0}^{k-1} \binom{n-k-1+j}{j}\,\frac{1}{j!}
	 \sum_{i=0}^{j} s(j,i)
	 \sum_{t=0}^{i}
	 \binom{i}{t}\,(n-k-1)^{t}\,j^{\,i-t}\,
	 (-1)^{\,k-1-t}.
	 \]
	 By the binomial theorem and Stirling identities:
	 \begin{eqnarray*}
	 \alpha_{k-1} &=& (-1)^{\,k-1}
	 \sum_{j=0}^{k-1} \binom{n-k-1+j}{j}\,\frac{1}{j!}
	 \sum_{i=0}^{j} s(j,i)\,\bigl(j-(n-k-1)\bigr)^{i}\\
	 &=&(-1)^{\,k-1}
	 \sum_{j=0}^{k-1} 
	 \binom{n-k-1+j}{j}\,\frac{1}{j!}\,
	 \bigl(j-(n-k-1)\bigr)_{j}\\
	 &=&(-1)^{\,k-1}
	 \sum_{j=0}^{k-1} 
	 \binom{n-k-1+j}{j}\,
	 \binom{\,j-(n-k-1)\,}{\,j\,}\\
	 &=&(-1)^{\,k-1}
	 \sum_{j=0}^{k-1} 
	 \binom{n-k-1+j}{j}\,
	 (-1)^{j}\binom{n-k-2}{j}.
	 \end{eqnarray*}

	 Since $n\leq k$, it makes sense to define $\Delta=k-n$. Rewriting we obtain:
	 \[
	 (-1)^{k-1}\sum_{j=0}^{k-1} (-1)^{j}\binom{-\Delta-1+j}{j}\binom{-\Delta-2}{j}
	 \]
	 Note that if $j>\Delta$ we have $\binom{-\Delta+1+j}{j}$ with a positive upper index but a larger lower index, so these terms will contribute 0 to our sum. So, we may stop our index at $\Delta$ and rewrite to remove the negative upper index. Standard manipulations yield
	 \[
	 \alpha_{k-1}=(-1)^{k-1}\sum_{j=0}^{\Delta} (-1)^{j}\binom{\Delta}{j}\binom{\Delta+j+1}{j}
	 = (-1)^{k-1}\sum_{j=0}^{\Delta} \binom{\Delta}{j}\binom{-\Delta-2}{j}
	 \]\[
	 = (-1)^{k-1}\sum_{j=0}^{\Delta} \binom{\Delta}{\Delta-j}\binom{-\Delta-2}{j}
	 = (-1)^{k-1}\binom{-2}{\Delta}
	 = (-1)^{k-1}(-1)^{\Delta}\binom{\Delta+1}{\Delta}.
	 \]
	 Where the second to last equality is attained by Vandermonde's identity. Thus,
	 \[\alpha_{k-1}=(-1)^{2k-n-1}(k-n+1)=(-1)^{n+1}(k-n-1)\]
	 as desired.
\end{proof}

	 For the next steps, it is convenient to emphasize the parameter dependence and write
	 \[
	 \alpha_{k-1}\equiv \alpha_{k-1}(k,n),
	 \]
	 since the expression for $\alpha_{k-1}$ involves both $k$ and $n$.
	 
\begin{lemma}\label{lemalphasym}For fixed $k$, $\alpha_{k-1}(k,n)$ is symmetric in $n$ about $n=k+1$.
    %or $\alpha_{k-1}(k,k+1+t)=\alpha_{k-1}(k,k+1-t)$ for all $t=0,\ldots , k$
	\end{lemma}
\begin{proof}

	 Thinking of $n$ as a perturbation away from $k+1$ gives $n=((k+1)+t)$ for some real $t$. This, in turn yields the relation $t=n-(k+1)$ substituting this into our equation for $\alpha_{k-1}(k,n)$ gives 
	 \[
	 \alpha_{k-1}(t)=\sum_{j=0}^{k-1} (-1)^j \binom{t+j}{j}\binom{t-1}{j}.
	 \]
	 Changing $t$ to $-t$ gives
	 \[
	 (-1)^{k-1} \sum_{j=0}^{k-1} (-1)^j \binom{-t+j}{j}\binom{-t-1}{j}
	 =
	 (-1)^{k-1} \sum_{j=0}^{k-1} (-1)^{3j} \binom{t+j}{j}\binom{t-1}{j},
	 \]
	 by the negative index theorem. This returns our original relation, demonstrating symmetry about $k+1$.
\end{proof}

%JS reusing the same \alpha notation for single- and multi-parameter coefficients may be confusing
	 We may now begin to reason about the sign of $\alpha_{k-1}(n)$ for $n\geq2k+1$ to complete the proof of Theorem \ref{thm:n-k-1}.
     \begin{proof}[proof of Theorem \ref{thm:n-k-1}]
	 By specializing Lemma \ref{lemalphakm1}, we have that if $n=1$ then $\alpha_{k-1}(k,n)=k$. Thus, by applying Lemma \ref{lemalphasym}, we get that if $n=2k+1$ then $\alpha_{k-1}(k,n)=k$, which is clearly a positive number. Next, from Lemma \ref{lemalphakm1} we have for all $n \in \lbrace1,2,...,k-1\rbrace$, $\operatorname*{sgn}(\alpha_{k-1}(k,n))=-\operatorname{sgn}(\alpha_{k-1}(n+1))$. Since $\alpha_{k-1}(k,n)$ is a sum of continuous functions, by the Intermediate Value Theorem $\alpha_{k-1}(k,n)$ must have a root somewhere in each of the open regions $(1,2),(2,3), ... ,(k-1,k)$. By application of Lemma \ref{lemalphasym} we may also conclude that $\alpha_{k-1}(k,n)$ has a root somewhere in the regions $(k+2,k+3),(k+3,k+4),...,(2k,2k+1)$ accounting for $2k-2$ total roots.

     But, the degree of $\alpha_{k-1}(k,n)$ is $2k-2$.  Thus, all roots of $\alpha_{k-1}(k,n)$ occur in the open region $(1,2k+1)$. So, since $\alpha_{k-1}(k,n)$ is positive when $n=2k+1$ we must have that $\alpha_{k-1}(k,n)$ is also positive for $n>2k+1$. 

	Finally, we may examine the coefficient of the $x^{n-2}(x+1)$ term in the magic positive basis expansion of our full polynomial.
    %JS converted to \coef notation for consistency
	 \begin{eqnarray*}
\lefteqn{\coef_{x^{n-2}(x+1)}\big(Q_{k,n}((n-k-1)x)S_{k,n}((n-k-1)x)\big)}\\
	 	 &=&\coef_{x^{n-k}}Q_{k,n}((n-k-1)x)\cdot
         \coef_{x^{k-2}(x+1)}S_{k,n}(n-k-1(x))+\\
	 &&\coef_{x^{n-k-1}(x+1)}Q_{k,n}((n-k-1)x)\cdot \coef_{x^{k-1}}S_{k,n}(n-k-1(x)).
%\lefteqn{[x^{n-2}(x+1)]\big(Q_{k,n}((n-k-1)x)S_{k,n}((n-k-1)x)\big)}\\
%	 	 &=&[x^{n-k}]Q_{k,n}((n-k-1)x)\cdot [x^{k-2}(x+1)]S_{k,n}(n-k-1(x))+\\
%	 &&[x^{n-k-1}(x+1)]Q_{k,n}((n-k-1)x)\cdot [x^{k-1}]S_{k,n}(n-k-1(x)).
     \end{eqnarray*}

	 From earlier discussion, the first term vanishes since $\coef_{x^{n-k}}Q_{k,n}((n-k-1)x)=0$, while the second term is a product of a negative number and a positive number. Hence
	 \[
	 \coef_{x^{n-2}(x+1)}\big(Q_{k,n}((n-k-1)x)S_{k,n}((n-k-1)x)\big)<0,
	 \]
	 so $D_{k,n}((n-k-1)x)$ fails to exhibit magic positivity for $1\leq k <\lfloor\frac{n}{2}\rfloor$.
     \end{proof}
     
 \section{The $n-k$ dilate of $D_{k,n}$}
In this section, we express the Ehrhart polynomial of the $n-k$ dilate of the base polytope of minimal matroid of dimension $n$, rank $k$, under the magic positive basis, thus establishing Theorem \ref{thm: asymptotic}. 

Similar to our analysis for $D_{k,n}((n-k-1)x)$, to derive the signs of the coefficients of $D_{k,n}((n-k)x)$ in the magic positive basis $\{x^i(x+1)^{d-i}\}_{i=0}^d$, we first rewrite the Ehrhart polynomial of the $n-k$ dilate of the base polytope in the following form.
\[D_{k,n}((n-k)x)\propto Q_{k,n}((n-k)x)\cdot S_{k,n}((n-k)x).\]
for $Q_{k,n}(x)$ and $S_{k,n}(x)$ defined in Equations (\ref{defQkn}) and (\ref{defSkn}), respectively.

%\begin{align*} 
%D_{k,n}((n-k)x) &=\frac{1}{{{n-1}\choose{k-1}}}{{(n-k)x+n-k}\choose {n-k}}\sum_{j=0}^{k-1}{{{n-k-1+j}\choose j}{{{(n-k)x+j}}\choose j}}
% \\ &\propto \binom{(n-k)x+n-k}{n-k}\sum_{j=0}^{k-1}\binom{n-k-1+j}{j}\binom{(n-k)x+j}{j}
%\end{align*}
%For convenience, define
%$$Q_{k,n}((n-k)x):=\binom{(n-k)x+n-k}{n-k}$$
%and
%$$S_{k,n}((n-k)x):=\sum_{j=0}^{k-1}\binom{n-k-1+j}{j}\binom{(n-k)x+j}{j}$$
%where
%$$D_{k.n}((n-k)x)\propto Q_{k,n}((n-k)x)\cdot S_{k,n}((n-k)x).$$
We will study the respective coefficients of each polynomial in the magic positive basis.

\begin{lemma}
\label{lemQkn}
The magic positive expansion of $Q_{k,n}((n-k)x)$ is given by
\[Q_{k,n}((n-k)x)=\sum_{s=0}^{n-k}\gamma_s x^s(1+x)^{n-k-s}\]
where $\gamma_s$ is given by
\[\gamma_s=\coef_{x^{s}}(\prod_{p=1}^{n-k}(1+(\frac{n-k}{p}-1)x)).\]
\end{lemma}
\begin{proof}
Let $m=n-k$. By specializing Proposition 5.2 in Konoike \cite{konoike2025magicpositivityehrhartpolynomials} to the $m$-dimensional standard simplex with dilation factor $m$, we obtain
\[
Q_{k,n}((n-k)x)=\binom{mx+m}{m}=\frac{1}{m!}\prod_{j=0}^{m-1}\left((m-j)(x+1)+jx\right).
\]
Relabeling the index by $p=m-j$, this becomes
\[
Q_{k,n}((n-k)x)=\prod_{p=1}^{m}\left((x+1)+\left(\frac{m}{p}-1\right)x\right).
\]
Now expand this product in the magic positive basis $\{x^s(x+1)^{m-s}\}_{s=0}^{m}$. To obtain the term $x^s(x+1)^{m-s}$, we must choose the $x$-summand from exactly $s$ of the $m$ factors and the $(x+1)$-summand from the remaining $m-s$ factors. Hence its coefficient is
\[
\gamma_s=\sum_{1\leq p_1<\cdots<p_s\leq m}\prod_{\ell=1}^{s}\left(\frac{m}{p_\ell}-1\right),
\]
which is exactly the coefficient of $x^s$ in
\[
\prod_{p=1}^{m}\left(1+\left(\frac{m}{p}-1\right)x\right).
\]
Therefore,
\[
Q_{k,n}((n-k)x)=\sum_{s=0}^{m}\gamma_s x^s(x+1)^{m-s}=
\sum_{s=0}^{n-k}\gamma_s x^s(1+x)^{n-k-s},
\]
as desired.
\end{proof}

Next, we turn our attention to the $S_{k,n}((n-k)x)$ term:
\begin{lemma}
\label{lemSkn}
The coefficients of $S_{k,n}((n-k)x)$ in the magic positive basis (Equation (\ref{eqdilS})) are
\[\alpha_r=\coef_{x^r}\left((1-x)^{k-1}\sum_{j=0}^{k-1}\binom{n-k-1+j}{j}\frac{1}{j!}\frac{1}{(1-x)^{j}}\prod_{i=0}^{j-1}\left(j-i+(i-j+n-k)x\right)\right).\]
\end{lemma}

\begin{proof}
By Lemma \ref{lemdilS} for $\sigma = (n-k)$, we have
\[\alpha_r=\sum_{j=0}^{k-1}\binom{n-k-1+j}{j}\frac{1}{j!}\sum_{i=0}^{j}s(j,i)\sum_{t=0}^{\min(i,r)}\binom{i}{t}(n-k)^t j^{i-t} (-1)^{r-t}\binom{k-1-t}{r-t}.\]

We now further manipulate this expression for $\alpha_r$. 
To simplify the upper bound of summation $\min(i,r)$, we claim that it can be equivalently replaced by $i$, since for all $t>r$ we have $r-t<0$ and $k-1-t\geq 0$, yielding $\binom{k-1-t}{r-t}=0$, meaning all those terms are annihilated.
By the above observation, we can derive
\[\alpha_r=\sum_{j=0}^{k-1}\binom{n-k-1+j}{j}\frac{1}{j!}\sum_{i=0}^{j}s(j,i)\sum_{t=0}^{i}\binom{i}{t}(n-k)^t j^{i-t} (-1)^{r-t}\binom{k-1-t}{r-t}.\]
Note that by Binomial Theorem
\begin{align*}(-1)^{r-t}\binom{k-1-t}{r-t} &=\coef_{x^{r-t}}((1-x)^{k-1-t})\\
&=\const_x(x^{-(r-t)}(1-x)^{k-1-t})\\
&=\const_x(x^{-r}(1-x)^{k-1}x^t(1-x)^{-t}).\\
\end{align*}
Substituting the above into the expression for $\alpha_r$, and manipulating with the constant term yields
\begin{align*}
\alpha_r&=\sum_{j=0}^{k-1}\binom{n-k-1+j}{j}\frac{1}{j!}\sum_{i=0}^{j}s(j,i)\sum_{t=0}^{i}\binom{i}{t}(n-k)^t j^{i-t} (-1)^{r-t}\binom{k-1-t}{r-t}\\
&=\sum_{j=0}^{k-1}\binom{n-k-1+j}{j}\frac{1}{j!}\sum_{i=0}^{j}s(j,i)\sum_{t=0}^{i}\binom{i}{t}(n-k)^t j^{i-t} \const(x^{-r}(1-x)^{k-1}x^t(1-x)^{-t})\\
&=\const_x(\sum_{j=0}^{k-1}\binom{n-k-1+j}{j}\frac{1}{j!}\sum_{i=0}^{j}s(j,i)\sum_{t=0}^{i}\binom{i}{t}(n-k)^t j^{i-t} x^{-r}(1-x)^{k-1}x^t(1-x)^{-t})\\
&=\const_x\left(x^{-r}(1-x)^{k-1}\sum_{j=0}^{k-1}\binom{n-k-1+j}{j}\frac{1}{j!}\sum_{i=0}^{j}s(j,i)\sum_{t=0}^{i}\binom{i}{t}(n-k)^t j^{i-t} x^t(1-x)^{-t}\right)\\
&=\const_x\left(x^{-r}(1-x)^{k-1}\sum_{j=0}^{k-1}\binom{n-k-1+j}{j}\frac{1}{j!}\sum_{i=0}^{j}s(j,i)\sum_{t=0}^{i}\binom{i}{t}(n-k)^t j^{i-t} x^t(1-x)^{-t}\right)\\
&=\const_x\left(x^{-r}(1-x)^{k-1}\sum_{j=0}^{k-1}\binom{n-k-1+j}{j}\frac{1}{j!}\sum_{i=0}^{j}s(j,i)\sum_{t=0}^{i}\binom{i}{t} j^{i-t} \left(\frac{(n-k)x}{1-x}\right)^t\right)\\
&=\coef_{x^r}\left((1-x)^{k-1}\sum_{j=0}^{k-1}\binom{n-k-1+j}{j}\frac{1}{j!}\sum_{i=0}^{j}s(j,i)\sum_{t=0}^{i}\binom{i}{t} j^{i-t} \left(\frac{(n-k)x}{1-x}\right)^t\right)\\
\end{align*}
By Binomial Theorem,
\[\alpha_r=\coef_{x^r}\left((1-x)^{k-1}\sum_{j=0}^{k-1}\binom{n-k-1+j}{j}\frac{1}{j!}\sum_{i=0}^{j}s(j,i)\left(j+\frac{(n-k)x}{1-x}\right)^i\right)\]
By definition of the signed Stirling numbers of the first kind and of the falling factorial,
\begin{align*}
\sum_{i=0}^{j} s(j,i)\left(j+\frac{(n-k)x}{1-x}\right)^i= \left(j+\frac{(n-k)x}{1-x}\right)_j=\prod_{i=0}^{j-1}\left(j+\frac{(n-k)x}{1-x}-i\right)
\end{align*}
\begin{align*}
=\prod_{i=0}^{j-1}\frac{(j-i)(1-x)+(n-k)x}{1-x}\
&= \frac{1}{(1-x)^j}\prod_{i=0}^{j-1}\left(j-i+(i-j+n-k)x\right).
\end{align*}
Thus, the expression for $\alpha_r$ is 
\[\alpha_r=\coef_{x^r}\left((1-x)^{k-1}\sum_{j=0}^{k-1}\binom{n-k-1+j}{j}\frac{1}{j!}\frac{1}{(1-x)^{j}}\prod_{i=0}^{j-1}\left(j-i+(i-j+n-k)x\right)\right).\]
\end{proof}
Combining the results of Lemma \ref{lemQkn} and \ref{lemSkn}, we obtain an expression for the whole Ehrhart polynomial:
\begin{align*}
D_{k,n}((n-k)x)&\propto Q_{k,n}((n-k)x) S_{k,n}((n-k)x)\\
&=\left (\sum_{s=0}^{n-k} \gamma_s x^s(1+x)^{n-k-s}\right)\cdot\left(\sum_{r=0}^{k-1}\alpha_rx^r(x+1)^{k-1-r}\right)
\end{align*}
where it is obvious that, since in both bases the coefficients of the polynomial products are expanded in the same manner, the multiplication of the two polynomials results in the discrete convolution of their respective coefficients:

\begin{align*}
\coef_{x^u(1+x)^{n-1-u}}(D_{k,n}((n-k)x))&\propto \coef_{x^u(1+x)^{n-1-u}}(Q_{k,n}((n-k)x) S_{k,n}((n-k)x))\\
=&\coef_{x^u(1+x)^{n-1-u}}\left (\sum_{s=0}^{n-k} \gamma_s x^s(1+x)^{n-k-s}\right)\cdot\left(\sum_{r=0}^{k-1}\alpha_rx^r(x+1)^{k-1-r}\right)\\
=&\sum_{a+b=u,0\leq a\leq n-k,0\leq b\leq k-1}\gamma_a\alpha_b
\end{align*}
\begin{align*}
=&\sum_{a+b=u,0\leq a\leq n-k,0\leq b\leq k-1}\coef_{x^{a}}(\prod_{p=1}^{n-k}(1+(\frac{n-k}{p}-1)x))\cdot\\
&\coef_{x^b} \left((1-x)^{k-1}\sum_{j=0}^{k-1}\binom{n-k-1+j}{j}\frac{1}{j!}\frac{1}{(1-x)^{j}}\prod_{i=0}^{j-1}\left(j-i+(i-j+n-k)x\right)\right)\\
=&\coef_{x^{u}}\left(\prod_{p=1}^{n-k}(1+(\frac{n-k}{p}-1)x)\right)\times\\
&(1-x)^{k-1}\sum_{j=0}^{k-1}\binom{n-k-1+j}{j}\frac{1}{j!}\frac{1}{(1-x)^{j}} \prod_{i=0}^{j-1}\left(j-i+(i-j+n-k)x\right).
\end{align*}
Letting
\begin{align*}
A(x)&:=\left(\prod_{p=1}^{n-k}(1+(\frac{n-k}{p}-1)x)\right) \text{ and}\\
B(x)&:=\left((1-x)^{k-1}\sum_{j=0}^{k-1}\binom{n-k-1+j}{j}\frac{1}{j!}\frac{1}{(1-x)^{j}}\prod_{i=0}^{j-1}\left(j-i+(i-j+n-k)x\right)\right),
\end{align*}
we have that the polynomial $D_{k,n}((n-k)x)$
is magic positive if and only if  all coefficients of 
\[A(x)\cdot B(x)\]
are positive.

We now turn to the asymptotic magic positivity of the $n-k$ dilate by showing that Konoike's conjecture holds for sufficiently large $n$.
\begin{lemma}\label{lem:gamma:asymp}
As $n$ approaches infinity, we eventually have
\[\gamma_s>0\]
for $0\leq s\leq n-k$.
\end{lemma}

\begin{proof}
As $n$ goes to infinity,
\[Q'_{k,n}((n-k)x)\sim \sum_{p=1}^{n-k}(1+\frac{n}{p}x)\]
and thus
\[\gamma_s\sim n^s e_s(1,\frac{1}{2},....,\frac{1}{n-k})>0\]
where $e_s$ is the elementary symmetric sum.
\end{proof}

\begin{lemma}\label{lem:beta:asymp}
As $n$ approaches infinity, we eventually have
\[\alpha_r>0\]
for all $0\leq r\leq k-1$.
\end{lemma}
\begin{proof}
Suppose $0\leq r<k$. We consider the asymptotic of $\alpha_r$ when $n$ approaches infinity.
The coefficients $\alpha_r$ are generated as coefficients of the polynomial
\[B(x)=\sum_{j=0}^{k-1}\binom{n-k-1+j}{j}\frac{1}{j!}(1-x)^{k-1-j}\prod_{i=0}^{j-1}\left(j-i+(i-j+n-k)x\right).\]
Reindexing by $m:=j-i$ we have
\[\prod_{i=0}^{j-1}\left(j-i+(i-j+n-k)x\right)=\prod_{m=1}^{j}\left(m+(-m+n-k)x\right)\sim \prod_{m=1}^j (m+nx).\]

By definition of the unsigned Stirling numbers of the first kind $\binomsq{i}{j}$ \cite[Chapter 6]{knuth:concrete},
\[\prod_{m=1}^j (m+nx)=\frac{1}{nx}(nx)^{\overline{j+1}}=\frac{1}{nx}\sum_{m=0}^{j+1}\binomsq{j+1}{m}n^{m}x^m=\sum_{m=0}^{j}\binomsq{j+1}{m+1}n^{m}x^{m}\]
since $\binomsq{j+1}{0}=0$.
By Binomial theorem,
\[(1-x)^{k-1-j}=\sum_{i=0}^{k-1-j}\binom{k-1-j}{i}(-x)^{i}.\]

Thus, rearranging the summations we have
\begin{align*}
B(x)&\sim \sum_{j=0}^{k-1}\frac{(n-k-1+j)_{j}}{j!}\frac{1}{j!}\sum_{i=0}^{k-1-j}\binom{k-1-j}{i}(-x)^{i}\sum_{m=0}^{j}\binomsq{j+1}{m+1}n^{m}x^{m}\\
&\sim \sum_{i\geq 0,j\geq 0,i+j\leq k-1}\sum_{m=0}^j\frac{n^{j+m}}{(j!)^2} \binom{k-1-j}{i}(-1)^{i}\binomsq{j+1}{m+1}x^{i+m}\\
&\sim\frac{1}{((k-1)!)^2}\sum_{m=0}^{k-1} \binomsq{k}{m+1}n^{k-1+m}x^{m}\\
\end{align*}
since for each fixed $0\leq d\leq k-1$ where $i+m=d$,  thus $i+j=(d-m)+j\leq k-1$, then we have the bound $j+m=(j-m)+2m\leq (k-1-d)+2d=k-1+d$ which is achieved at $j=k-1$, $m=d$, $i=0$, yielding the dominant term of the coefficient of $x^d$. Thus, when $n$ approaches infinity,
\[\alpha_r\sim \frac{n^{k-1+m}}{((k-1)!)^2}\binomsq{k}{m+1}>0\]
for all $0\leq r\leq k-1$. 

\end{proof}

We are now ready to collect these results.

\begin{proof}[proof of Theorem \ref{thm: asymptotic}]
It follows directly from Lemma \ref{lem:gamma:asymp} and Lemma \ref{lem:beta:asymp}, that the Ehrhart polynomial $D_{k,n}((n-k)x)$ is magic positive for sufficiently large $n$. By our work in Section \ref{secnk1}, the $n-k-1$ dilate of $D_{k,n}$ is not magic positive for $1\leq k<\lfloor\frac{n}{2}\rfloor$. Then, we can conclude that Konoike's conjecture holds for sufficiently large $n$, proving Theorem \ref{thm: asymptotic}. 
\end{proof}

%added a bit more, so it makes sense to have it as a separate section
\section{Numerical Results}

Theorem \ref{thm: sharp} can be shown through brute-force with the help of a symbolic calculation software (such as Maple). It suffices to expand the expression
\begin{align*}
&\left( \prod_{i=1}^{n-k}\left( iy+\left( n-k-i\right) x\right) \right) \sum_{j=0}^{k-1}\binom{n-k-1+j}{j}\left( y-x\right) ^{k-1-j}\times \\
&\frac{1}{j!}\left( \prod_{\ell =1}^{j}\left( \ell y+\left( n-k-\ell \right) x\right) \right)
\end{align*}
for specific values of $n$ and $k$, as a polynomial in $y$ and $x$ and verify if each coefficient is positive. (Here, $y$ stands for $(x+1)$ --- this relabeling guides the symbolic engine towards remaining in the magic positive basis.) We report here the lowest order example, which is sufficiently simple to be verifiable manually.
\begin{example}
Note that for $n=3, k=2$,
\[
D_{2,3}((3-2)x) \propto 2(x+1)^2-(x+1)x,
\]
while for $n=3,k=1$,
\[
D_{1,3}((3-1)x) \propto (x+1)^2+(x+1)x.
\]
% So it appears we don't have to go to n=13 to make our point.
% Or did I misunderstand something?
% Note that when $n=13,k=7$,
% $$\left(\prod_{p=1}^{n-k}(1+(\frac{n-k}{p}-1)x)\right)\cdot\left((1-x)^{k-1}\sum_{j=0}^{k-1}\binom{n-k-1+j}{j}\frac{1}{j!}\frac{1}{(1-x)^{j}}\prod_{i=0}^{j-1}\left(j-i+(i-j+n-k)x\right)\right)$$
% $$=-x^{11}-\frac{2431}{10}x^{10}-\frac{89529}{50}x^9+4785 x^8+\frac{4125099}{50}x^7+\frac{7754439}{25}x^6+\frac{5612211}{10}x^5+\frac{13933953}{25}x^4$$
% $$+\frac{15644607}{50}x^3+\frac{4848833}{50}x^2+\frac{75773}{5}x+924.$$
% %JS: what is 'x' in the above expression? The same x as in D_{k,n}(x) ?? -- No, replace x^m :=x^m(x+1)^{n-k-m}
% %JS: OK!
% %JS: I don't see how below follows from just above... k=7 vs k=8
% By our previous result, this implies that
% $$D_{8,13}((13-8)x)\propto Q_{k,n}((n-k)x) S_{k,n}((n-k)x)$$
% $$=-x^{11}(x+1)-\frac{2431}{10}x^{10}(x+1)^2-\frac{89529}{50}x^9(x+1)^3+4785 x^8(x+1)^4+\frac{4125099}{50}x^7(x+1)^5$$
% $$+\frac{7754439}{25}
% x^6(x+1)^6+\frac{5612211}{10}x^5(x+1)^7+\frac{13933953}{25}x^4(x+1)^8
% +\frac{15644607}{50}x^3(x+1)^9$$
% $$+\frac{4848833}{50}x^2(x+1)^{10}+\frac{75773}{5}x(x+1)^{11}+924(x+1)^{12}.$$
% Since there are several negative coefficients, by definition, $D_{8,13}((13-8)x)$ is not magic positive.
% \\
% %JS: rephrased to be clear OK
Thus, magic positivity does not hold for $k>\lfloor\frac{n}{2}\rfloor$ in general, and the bound on $k$ provided by Konoike's conjecture is therefore sharp.

\end{example}

\section{Conclusion}
In this paper, we have analyzed the magic positivity of the Ehrhart polynomial
of dilated base polytopes of minimal matroids, working from the known expression
for the lattice point enumerator $D_{k,n}(x)$ due to Konoike
\cite{konoike2025magicpositivityehrhartpolynomials}.  Our results provide
substantial progress toward a complete answer to Question~1.1.

Theorem~\ref{thm:n-k-1} establishes that the $m$-index of $B(T_{k,n})$ is at
least $n-k$ for $1 \le k < \lfloor n/2 \rfloor$, by proving that the
$(n-k-1)$-th dilate always fails magic positivity.  This gives a sharp lower
bound on the $m$-index throughout this range, confirming one half of the
conjecture.  Theorem~\ref{thm: asymptotic} establishes the conjecture in the
asymptotic regime: for any fixed $k$, the $(n-k)$-th dilate is magic positive
for all sufficiently large $n$.  Together, these results show that the
conjectured value $n-k$ is both a lower bound and (asymptotically) an upper
bound for the $m$-index.  Theorem~\ref{thm: sharp} provides complementary
numerical evidence, showing that the bound $k \le \lfloor n/2 \rfloor$ in the
conjecture is necessary: for $k > \lfloor n/2 \rfloor$, magic positivity fails
at the $(n-k)$-th dilate even for small values of $n$.

What remains open is the exact determination of $m\text{-index}(B(T_{k,n})) =
n-k$ for all $1 \le k \le \lfloor n/2 \rfloor$ and all $n$, not just
asymptotically.  A complete proof would require showing that all
magic-positive-basis coefficients of $D_{k,n}((n-k)x)$ are non-negative for all
$n$ in this range, a task that appears to require sharper control of the
interaction between the $Q_{k,n}$ and $S_{k,n}$ components than our current
asymptotic methods provide.

Our results suggest several natural directions for future work.  First, the
methods developed here --- in particular, the factorization of $D_{k,n}$ and
the use of Stirling-number identities to track sign behavior --- may apply to
Ehrhart polynomials of more general families of matroid base polytopes.
Ferroni \cite{Ferroni_2021} conjectures that similar positivity phenomena hold
for broader classes of matroids, and it would be interesting to investigate
whether the $m$-index can be computed or bounded for those families using
analogous techniques.  Second, the magic positivity framework itself could be
explored for other combinatorially defined polytopes, such as order polytopes,
flow polytopes, or generalized permutohedra, where the Ehrhart theory is
well-developed but the $h^*$-polynomial structure is less well understood.

\section*{Acknowledgement}

The authors are thankful to Shen Harman and Maxwell Dodson, students of Stanford Online High School, for helpful conversations.

%JS: in intro now
%\nocite{beck:disc}

\bibliography{sn-bibliography.bib}% common bib file

\end{document}